\documentclass[12pt,reqno]{amsart}
\usepackage{a4wide}
\usepackage[utf8]{inputenc}
\usepackage{lmodern}
\usepackage{amsmath}
\usepackage{amssymb}
\usepackage{amsthm}
\usepackage{mathrsfs}
\usepackage{enumerate}
\usepackage{stmaryrd,newlfont}

\usepackage{graphicx}
\usepackage{color}
\usepackage{tikz}
\usetikzlibrary{calc,shapes.arrows,decorations.pathreplacing}

\newcommand{\N}{\mathbb{N}}
\newcommand{\Z}{\mathbb{Z}}

\newcommand{\R}{\mathbb{R}}
\newcommand{\C}{\mathbb{C}}

\newcommand{\T}{\mathbb{T}}
\newcommand{\D}{\mathcal{D}}
\newcommand{\DT}{\mathcal{D}'(\T)}
\newcommand{\MT}{\mathcal{M}(\T)}

\newcommand{\LL}{\ell^p}
\newcommand{\A}[1]{A^{#1}(\mathbb{T})}

\newcommand{\supp}{\mathrm{supp}}

\newcommand{\PT}{\mathcal{P}(\T)}

\newcommand{\PTp}{\mathcal{P}_+(\T)}
\newcommand{\Lip}{\mathrm{Lip}}
\newcommand{\cZ}{{\mathcal{Z}}}

\title[Cyclicity in $\ell^p$ spaces and zero sets of the Fourier transforms]{Cyclicity in {\large$\LL$} spaces and zero sets of the Fourier transforms}
\author[F. Le Manach]{Florian Le Manach}
\date{\today}
\address{Institut de Math\'ematiques de Bordeaux, Universit\'e de Bordeaux/Bordeaux INP/CNRS, 351 Cours de la
Lib\'eration, 33 405 Talence, France.}
\email{florian.le-manach@math.u-bordeaux.fr}

\keywords{Cyclicity, weighted $\ell^p$ spaces} 
\subjclass[2000]{primary 47A16; secondary 47B37, 43A15, 31C15.}
\newtheorem{Def}{Definition}[section]
\newtheorem{lemme}[Def]{Lemma}
\newtheorem{prop}[Def]{Proposition}
\newtheorem{theo}[Def]{Theorem}
\newtheorem{rem}[Def]{Remark}

\numberwithin{equation}{section}

\begin{document}

\begin{abstract}
We study the cyclicity of vectors $u$ in $\ell^p(\Z)$. It is known that a vector $u$ is cyclic in $\ell^2(\Z)$ if and only if the zero set, $\cZ(\widehat{u})$, of its Fourier transform, $\widehat{u}$, has Lebesgue measure zero and $\log |\widehat{u}| \not \in L^1(\T)$, where $\T$ is the unit circle. Here we show that, unlike $\ell^2(\Z)$, there is no characterization of the cyclicity of $u$ in $\ell^p(\Z)$, $1<p<2$, in terms of $\cZ(\widehat{u})$ and the divergence of the integral $\int_\T \log |\widehat{u}| $.
Moreover we give both necessary conditions and sufficient conditions for $u$ to be cyclic in $\ell^p(\Z)$, $1<p<2$.
\end{abstract}

\maketitle

\section{Introduction and main results}

In this paper we will study cyclic vectors in $\ell^p$ spaces for $1<p<2$.
A vector $u=(u_n)_{n \in \Z} \in \ell^p(\Z)$ is called {\it cyclic} in $\ell^p(\Z)$ if the linear span of $ \{(u_{n-k})_{n \in \Z},~ k = 0,1,2,\dots \}$ is dense in $\ell^p(\Z)$ and $u$ is called {\it bicyclic} in $\ell^p(\Z)$ if the linear span of $ \{(u_{n-k})_{n \in \Z},~ k \in \Z \}$ is dense in $\ell^p(\Z)$. Note that the bicyclicity in this paper is defined as cyclicity in \cite{LMF} and \cite{LEV}.

\bigskip
We denote by $\T$ the unit circle of $\C$. The Fourier transform of $u=(u_n)_{n \in \Z} \in \ell^p(\Z)$ is given by
$$\widehat{u} : \zeta \in \T \mapsto \sum_{n\in \Z}u_n \zeta^n$$
and we denote by $\cZ(\widehat{u})$ the zero set of $\widehat{u}$ on $\T$:
$$\cZ(\widehat{u})=\{\zeta \in \T,~ \widehat{u}(\zeta)=0\}.$$
Note that if $u \in \ell^1(\Z)$ then $\widehat{u}$ is continuous and $\cZ(\widehat{u})$ is well defined, while if $u \in \ell^2(\Z)$ then $\widehat{u} \in L^2(\T)$ and $\cZ(\widehat{u})$ is defined modulo null sets.

\bigskip
Wiener showed in \cite{WIE} that the bicyclicity of a sequence $u$ in $\ell^1(\Z)$ and $\ell^2(\Z)$ is characterized only in terms of $\cZ(\widehat{u})$: $u$ is bicyclic in $\ell^1(\Z)$ if and only if $\cZ(\widehat{u})=\varnothing$ and $u$ is bicyclic in $\ell^2(\Z)$ if and only if $\cZ(\widehat{u})$ has Lebesgue measure zero. Wiener asked if the same phenomenon holds for $u \in \ell^p(\Z)$, $1<p<2$.

Lev and Olevskii, in \cite{LEV}, answered this question negatively. We cannot characterize the bicyclicity of $u \in \ell^1(\Z)$ only in terms of the zero set of $\widehat{u}$: for $1<p<2$ there exist $u,v \in \ell^1(\Z)$ such that $\cZ(\widehat{u})=\cZ(\widehat{v})$ and one is bicyclic in $\ell^p(\Z)$ and the other is not. Here we consider a similar problem for cyclicity in $\ell^p(\Z)$, $1<p<2$.

\bigskip
It follows from Wiener characterization of bicyclic vectors and Szeg\"o infimum that a vector $u$ is cyclic in $\ell^2(\Z)$ if and only if $\cZ(\widehat{u})$ has Lebesgue measure zero and $\log |\widehat{u}| \not \in L^1(\T)$ (see \cite[Corollary 4.1.2, Theorem 4.1.4]{NIK}). Furthermore when $u$ is cyclic in $\ell^p(\Z)$ then $u$ is cyclic in $\ell^q(\Z)$ for $q \geq p$. A direct consequence of this fact is that, for $u \in \ell^1(\Z)$, if $u$ is cyclic in $\ell^p(\Z)$, $1 \leq p \leq 2$, then $\cZ(\widehat{u})$ has Lebesgue measure zero and $\log |\widehat{u}| \not \in L^1(\T)$. It follows from this and Wiener characterization of bicyclic vectors in $\ell^1(\Z)$ that there is no cyclic vectors in $\ell^1(\Z)$ (see also \cite{ABA} for a different proof).

\bigskip
The existence of cyclic vectors on $\ell^p(\Z)$, for $p > 1$, was established by Abakumov, Atzmon and Grivaux in \cite[Corollary 1]{ABA} in 2007. Note that this result was first proved by Olevskii in an unpublished manuscript in 1998 (see \cite[Remark 1]{ABA}).

\bigskip
We show that for sequences $u \in \ell^1(\Z)$ such that $\log |\widehat{u}| \not \in L^1(\T)$, there is no characterization of the cyclic vectors $u$ in $\ell^p(\Z)$ only in terms of $\cZ(\widehat{u})$.

\begin{theo} \label{levsuite}
Let $1 < p < 2$. There exist $u$ and $v$ in $\ell^1(\Z)$ such that $\cZ(\widehat{u})=\cZ(\widehat{v})$, $\log |\widehat{u}| \not \in L^1(\T)$, $\log |\widehat{v}| \not \in L^1(\T)$, with $u$ non-cyclic in $\ell^p(\Z)$ and $v$ cyclic in $\ell^p(\Z)$.
\end{theo}

Until now we did not know more than the existence of cyclic vectors in $\ell^p(\Z)$. Here we give, when $\widehat{u}$ is smooth, both necessary conditions and sufficient conditions for $u$ to be cyclic in $\ell^p(\Z)$, $1<p<2$, and in a more general case: the weighted $\ell^p$ spaces.  In particular this allows to construct explicit examples of cyclic vectors in weighted $\ell^p$ spaces. Here are some results that we get: for $u \in \ell^1(\Z)$ and $1<p<2$,

\begin{itemize}
\item if $\widehat{u} \in C^\infty(\T)$, $\dim(\cZ(\widehat{u})) < 2(p-1)/p$ and $\log |\widehat{u}| \not \in L^1(\T)$ then $u$ is cyclic in $\ell^p(\Z)$,
\item if $\widehat{u} \in \text{Lip}_\delta(\T)$ for $\delta > 1/p-1/2$, $\dim(\cZ(\widehat{u})) < 2(p-1)/p$ and $\log (d(\cdot,\cZ(\widehat{u}))) \not \in L^1(\T)$ then $u$ is cyclic in $\ell^p(\Z)$,
\end{itemize}
where $\dim$ is the Hausdorff dimension (see \cite[Chap II]{KAH}), $d(\cdot,\cZ(\widehat{u}))$ is the Euclidean distance to $\cZ(\widehat{u})$ and $f \in  \text{Lip}_\delta(\T)$ means that there exists $C > 0$ such that 
\begin{equation}\label{Lipdef}
|f(\zeta)-f(\zeta')| \leq C |\zeta-\zeta'|^\delta,\qquad  \zeta,\zeta' \in \T.  
\end{equation}

Let $E$ be a closed subset of $\T$. The set $E$ is called a Carleson set if $\log d(\cdot,E)\in L^1(\T)$ (see \cite{CAR1}).  This condition is equivalent to the fact that $E$ has a Lebesgue measure zero  and $\sum_n |I_n|\log |I_n|>-\infty$, where $(I_n)$ are the complementary intervals of $E$. Note that  if $\widehat{u} \in \text{Lip}_\delta(\T)$ such that $\cZ(\widehat{u})$ is not a Carleson set then $\log |\widehat{u}|\notin L^1(\T)$.

\bigskip
The paper is organized as follows: in section 2 we give preliminary results, in section 3 we prove Theorem \ref{levsuite} and section 4 is devoted to establishing both necessary conditions and sufficient conditions for cyclicity in weighted $\ell^p(\Z)$.

\bigskip

We use the following notations:
\begin{itemize}
\item $A\lesssim B$ means that there is an absolute constant $C$ such that $A \le CB$. 
\item $A\asymp B$  if both $A\lesssim B$ and $B\lesssim A$. 
\end{itemize}

\section{Preliminaries}

We denote by $\N_0$ the set of all non-negative integers and by $\N$ the set of all positive integers.
Let $X$ be a metric linear space of complex functions on $\T$ such that the shift operator $S$, given by
$$S(f)(\zeta) = \zeta f(\zeta)$$
is a topological isomorphism of $X$ onto itself. Let $f \in X$ and $\Lambda \subset \Z$, we denote by $[f]_\Lambda^X$ the closed linear span of $\{ z^nf,~ n \in \Lambda \}$. We say that $f \in X$ is {\it cyclic} in $X$ if $[f]_{\N_0}^X=X$ and $f$ is {\it bicyclic} in $X$ if $[f]_\Z^X=X$. Note that $f$ is cyclic in $X$ if and only if $f$ is bicyclic in $X$ and $[f]_{\N_0}^X=[f]_\Z^X$. Moreover
\begin{equation} \label{equitop}
[f]_{\N_0}^X=[f]_\Z^X \Longleftrightarrow f \in [f]_{\N}^X
\end{equation}

Indeed, if $[f]_{\N_0}^X=[f]_\Z^X$ then we have $\overline{z}f \in [f]_{\N_0}^X$ and we obtain that $f \in [f]_{\N}^X$, since the multiplication by $z$ is continuous in $X$. Conversely, we assume that $f \in [f]_{\N}^X$. Then we see easily that $\overline{z}f \in [f]_{\N_0}^X$. Moreover if $\overline{z}^nf \in [f]_{\N_0}^X$ for some $n \geq 1$ then $\overline{z}^{n+1}f = \lim \overline{z} P_kf$ with $P_k \in \PTp$. Since $P_k \overline{z}f \in [f]_{\N_0}^X$ and $[f]_{\N_0}^X$ is closed in $X$, we obtain $\overline{z}^{n+1}f \in [f]_{\N_0}^X$. Thus by induction, for all $n \geq 1$, $\overline{z}^nf \in [f]_{\N_0}^X$ and then $[f]_{\N_0}^X=[f]_\Z^X$.

\bigskip
We denote by $\PT$ the set of trigonometric polynomials on $\T$ and by $\PTp$ the set of analytic polynomials on $\T$. Now we suppose, in addition, that $X$ is a Banach space such that $\PT$ is dense in $X$. Then $f \in X$ is bicyclic in $X$ if and only if there exists a sequence $(P_n)$ of trigonometric polynomials such that
\begin{equation}\label{caraCyclNorm}
\lim_{n \to \infty} \|1-P_nf\|_{X} = 0.
\end{equation}

So we obtain, from this and \eqref{equitop}, the following proposition:

\begin{prop} \label{base}
Let $f \in X$. The function $f$ is cyclic in $X$ if and only if there exist $P_n \in \PT$ and $Q_n \in \PTp$ satisfying 
$$\|1-P_n f \|_{X}  \underset{n \to \infty}{\longrightarrow} 0 \quad \text{ and } \quad \|f-zQ_n f \|_{X} \underset{n \to \infty}{\longrightarrow} 0.$$
\end{prop}

Now we study the weighted $\ell^p$ spaces. We denote by $\DT$ the set of distributions on $\T$ and $\MT$ the set of measures on $\T$. For $S \in \DT$ and $n \in \Z$, we denote by $\widehat{S}(n)$ the $n$th Fourier coefficient of $S$ defined by $\widehat{S}(n) = \langle S, z^{-n} \rangle$ and we write $S = \sum_{n \in \Z} \widehat{S}(n) z^n$. If $f \in L^1(\T)$ we have
$$\widehat{f}(n)=  \int_\T f(\zeta)\zeta^{-n} \frac{|d \zeta |}{2\pi}.$$
For $p \geq 1$ and $\beta \in \R$, we define the Banach spaces
$$\ell^p_\beta(\Z) = \Big\{ (u_n)_{n \in \Z} \in \C^\Z, \;\; 
 \|(u_n)\|_{\ell^p_\beta}^p = \sum_{n\in \Z} |u_n|^p (1+|n|)^{p\beta} < \infty \Big\}$$
 and
$$A^p_\beta(\T) = \Big\{ S \in \DT, \;\; 
 \|S\|_{A^p_\beta}^p = \sum_{n\in \Z} |\widehat{S}(n)|^p (1+|n|)^{p\beta} < \infty \Big\}.$$
We will write $A^p(\T)$ for the space $A^p_0(\T)$ and $A(\T)$ for the space $A^1(\T)$ which is the Wiener algebra. The Fourier transformation is an isometric isomorphism between $\ell^p_\beta(\Z)$ and $A^p_\beta(\T)$ since for $S \in A^p_\beta(\T)$ we have
$$\|S\|_{A^p_\beta} = \|(\widehat{S}(n))_{n \in \Z}\|_{\ell^p_\beta}.$$
Moreover, for a sequence $(u_n)_{n \in Z}$ and $k \in \Z$, we have $\widehat{(u_{n-k})} = z^k \widehat{(u_n)}$, so $u$ is cyclic in $\ell^p_\beta(\Z)$ if and only if $\widehat{u}$ is cyclic in $A^p_\beta(\T)$. From now we will state and prove our results in $A^p_\beta(\T)$.  Note that if $1 \leq p \leq 2$ and $\beta \geq 0$, $A^p_\beta(\T)$ is a subset $L^2(\T)$.

\bigskip
For $1 \leq p < \infty$ and $\beta \geq 0$ we define the product of $f \in A^1_\beta(\T)$ and $S \in A^p_\beta(\T)$ by
$$fS = \sum_{n \in \Z} \Big( \sum_{k\in \Z} \widehat{f}(k)\widehat{S}(n-k) \Big) z^n,$$
and we see that $\|fS\|_{A^p_\beta} \leq \|f\|_{A^1_\beta} \|S\|_{A^p_\beta}$. 

\bigskip
The following lemma gives us different inclusions between the $A^p_\beta(\T)$ spaces.

\begin{lemme}[\cite{LMF}] \label{inclusionAp}
Let $1 \leq r,s < \infty$ and $\beta,\gamma \in \R$.
\begin{enumerate}
\item If $r \leq s$ then $A^r_\beta(\T) \subset A^s_\gamma(\T) \iff \gamma \leq \beta$.
\item If $r>s$ then $A^r_\beta(\T) \subset A^s_\gamma(\T) \iff \beta - \gamma > \frac{1}{s} - \frac{1}{r}$.
\end{enumerate}
\end{lemme}

\bigskip
For $p \neq 1$, the dual space of $A^p_\beta(\T)$ can be identified with $A^q_{-\beta}(\T)$ ($q=\frac{p}{p-1}$) by the following formula
$$\langle S,T \rangle = \sum_{n \in \Z} \widehat{S}(n) \widehat{T}(-n), \qquad S \in A^p_\beta(\T),~ T \in A^q_{-\beta}(\T).$$

A distribution $S$ on $\T$ is said to vanish on an open set $V$, if $\langle S,\varphi \rangle=0$ for all functions $\varphi \in C^\infty(\T)$  with support in $V$. The support of the distribution $S$, denoted by $\supp(S)$, is the complement of the maximal open set on which the distribution $S$ vanishes. When $f$ is continuous, we denote by $\cZ(f)$ the zero set of the function $f$. For $\delta \geq 0$ we recall that $\Lip_\delta(\T)$ denotes the set of functions $f$ defined on $\T$ which satisfies \eqref{Lipdef}.

We need the following proposition which gives a necessary condition for a function $f \in \Lip_\delta(\T)$ to be bicyclic in $A^p(\T)$.

\begin{prop}[\cite{HER}] \label{regCycl} Let $1 < p < \infty$, $q = p/(p-1)$ and $f \in A(\T)$. If there exists $\delta > 0$ such that $f \in \Lip_\delta(\T)$ and if there exists $S \in A^q(\T) \setminus \{0\}$ such that $\supp(S) \subset \cZ(f)$ then $f$ is not bicyclic in $A^p(\T)$.
\end{prop}

\section{No characterization of cyclicity in terms of zero set}

The following Theorem is a reformulation of Theorem \ref{levsuite} in $A^p(\T)$.

\begin{theo}\label{levmainresult}
Let $1 < p < 2$. There exist $f$ and $g$ in $A(\T)$ such that $\cZ(f)=\cZ(g)$, $ \log |f| \not \in L^1(\T)$, $\log |g| \not \in L^1(\T)$ with $f$ non-cyclic in $A^p(\T)$ and $g$ cyclic in $A^p(\T)$.
\end{theo}

For the proof we need some notations and definitions. Let $E$ be a closed subset of $\T$. We set $$\mathcal{I}(E)=\{ f \in A(\T),~ f|_E=0 \}$$ and we denote by $\mathcal{J}(E)$ the closure in $A(\T)$ of the set of functions in $A(\T)$ which vanish on a neighborhood of $E$. Clearly $\mathcal{J}(E) \subset \mathcal{I}(E)$. We say that $E$ is a set of synthesis if $\mathcal{J}(E) = \mathcal{I}(E)$. Since $\{ 1 \}$ is a set of synthesis (see \cite[Theorem IV]{KAH} p. 123 or \cite[Appendix 3]{GRA} pp. 416-418), we have $\mathcal{I}(\{ 1 \})=\mathcal{J}(\{ 1 \})$. Unlike the space $A(\T)$, $\mathcal{I}(\{ 1 \})$ contains cyclic vectors. Moreover, by \cite[Proposition 2]{ABA}, the set of cyclic vectors in $\mathcal{I}(\{ 1 \})$ is a dense $G_\delta$ subset of $\mathcal{I}(\{ 1 \})$.

\begin{lemme} \label{cycliciteA0Ap}
Let $f \in \mathcal{I}(\{ 1 \})$. If $f$ is cyclic in $\mathcal{I}(\{ 1 \})$ then $f$ is cyclic in $A^p(\T)$ for all $p>1$.
\end{lemme}

\begin{proof}
Let $f \in \mathcal{I}(\{ 1 \})$ cyclic in $\mathcal{I}(\{ 1 \})$ and $\varepsilon > 0$. We denote by $h_k$, for $k \in \N$, the function defined on $\T$ by $$h_k(z)=\frac{z-1}{z-1-1/k}.$$
Clearly $h_k \in \mathcal{I}(\{ 1 \})$. Moreover for $z \in \T$,
$$h_k(z) = 1 - \frac{1}{k+1} \sum_{n=0}^\infty \frac{z^n}{(1+1/k)^n}.$$
Thus for $p > 1$, we have
\begin{eqnarray*}
\| 1 - h_k \|_{A^p}^p &=& \frac{1}{(k+1)^p} \sum_{n=0}^\infty \frac{1}{(1+1/k)^{pn}}\\
&=& \frac{1}{(k+1)^p} \dfrac{1}{1-\frac{1}{(1+1/k)^p}}\\
&=& \frac{1}{(k+1)^p-k^p} ~~\underset{k \to \infty}{\sim}~~ \frac{1}{pk^{p-1}}.
\end{eqnarray*}
So for fixed $p>1$, we choose $k \in \N$ such that $\| 1 - h_k \|_{A^p} < \varepsilon$.
Since $f$ is cyclic in $\mathcal{I}(\{ 1 \})$, there exist $P$ and $Q$ in $\PTp$ such that
$$\|h_k-Pf\|_{A^1} < \varepsilon ~~~\text{ and }~~~ \|f-zQf\|_{A^1} < \varepsilon.$$
We obtain that $\|f-zQf\|_{A^p} < \varepsilon$ and
$$\|1-Pf\|_{A^p} \leq \| 1 - h_k \|_{A^p} + \|h_k-Pf\|_{A^1} < 2\varepsilon.$$
This proves, by Proposition \ref{base}, that $f$ is cyclic in $A^p(\T)$ for all $p>1$.
\end{proof}

To prove our main theorem of this section, we need to introduce the notion of Helson sets.

\begin{Def}[\cite{KAH}, Chap. XI] \label{Helson}
A compact set $K \subset \T$ is called a Helson set if it satisfies any one of the following equivalent conditions:
\begin{enumerate}[(i)]
\item if $f \in \mathcal{C}(K)$ then there exists $g \in \A{}$ such that $g|_K = f$,
\item there exists $\delta_1 > 0$ such that for all $\mu \in \MT$ satisfying $\supp(\mu) \subset K$ we have
$$\underset{n\in\Z}{\sup} ~ |\hat{\mu}(n)| \geq \delta_1 \int_\T |\mathrm{d} \mu|,$$
\item there exists $\delta_2 > 0$ such that for all $\mu \in \MT$ satisfying $\supp(\mu) \subset K$ we have
$$\underset{|n|\to\infty}{\limsup}~ |\hat{\mu}(n)| \geq \delta_2 \int_\T |\mathrm{d} \mu|.$$
\end{enumerate}
\end{Def}

By $(iii)$, a Helson set does not support any non zero measure $\mu$ such that $\widehat{\mu}(n) \to 0$ as $|n| \to \infty$. On the other hand K\"orner shows, in \cite{KOR}, that there exists a Helson set which support a non zero distribution $S$ such that $\widehat{S}(n) \to 0$ as $|n| \to \infty$ (see also Kaufman in \cite{KAU} for a simple proof). Lev and Olevskii extend significantly this result by showing the following theorem which is the key of their result about bicyclicity in $\ell^p(\Z)$.

\begin{theo}[\cite{LEV}, Theorem 3] \label{LevHelson}
For $q > 2$, there exists a Helson set $K$ which supports a non-zero distribution $S \in A^q(\T)$.
\end{theo}

We need also the following lemma obtained by Lev and Olevskii.

\begin{lemme}[\cite{LEV}, Lemma 10] \label{lemme10}
Let $K$ be a Helson set. For $\varepsilon >0$, $p > 1$ and $f \in \mathcal{C}(K)$, there exists $g \in \A{}$ such that
$$\left\{
 \begin{array}{rcl} 
 \displaystyle g|_K &=& f, \\ 
 \displaystyle ||g||_{A^1} &\leq&  \displaystyle\frac{1}{\delta_2} ||f||_\infty, \\ 
 \displaystyle ||g||_{A^p} &<& \varepsilon
  \end{array}
   \right. $$
where $\delta_2$ is a constant which satisfies $(iii)$ in the definition \ref{Helson}.
\end{lemme}

The main result, Theorem \ref{levmainresult}, is a consequence of the following theorem which is an analogue of Lev and Olevskii result. 

\begin{theo} \label{Existence}
Let $K$ be a Helson set on $\T$. There exists $g \in A(\T)$ such that $g$ vanishes on $K$ and such that $g$ is cyclic in $A^p(\T)$ for all $p > 1$.
\end{theo}

\begin{proof}
Let $K$ be a Helson set on $\T$. We recall that
$$\mathcal{I}(K)=\{ g \in \A{},~ g|_K = 0 \}$$
which, endowed with the norm $\|\cdot\|_{A^1}$, is a Banach space. We define
$$\mathcal{G}(K) = \{ g \in \mathcal{I}(K),~ g \text{ is cyclic in } \A{p}, ~ \forall p > 1 \}.$$
We will show that $\mathcal{G}(K)$ is a dense $G_\delta$ subset of $\mathcal{I}(K)$.
For $\varepsilon > 0$ and $p > 1$, we consider the set
$$
G(\varepsilon,p) = \{ g \in \mathcal{I}(K),~ \exists P \in \PT,~ \exists Q \in \PTp,~ \|1-Pg\|_{A^p} < \varepsilon \text{ and }~ \|g-zQg\|_{A^p} < \varepsilon \}.
$$
Let $\varepsilon_n = \frac{1}{n}$ and $p_n=1+\frac{1}{n}$. To use a Baire Category argument, we will show that
$$\begin{array}{ll} (i) & \underset{n=1}{\overset{\infty}{\bigcap}} G(\varepsilon_n,p_n) = \mathcal{G}(K), \\ (ii) & \text{for $\varepsilon > 0$ and $p > 1$},~~ G(\varepsilon,p) \text{ is an open subset of } \mathcal{I}(K), \\ (iii) & \text{for $\varepsilon > 0$ and $p > 1$},~~ G(\varepsilon,p) \text{ is a dense subset of } \mathcal{I}(K). \end{array}$$

\bigskip \noindent
$(i)$: Let $g \in \mathcal{I}(K)$ such that for all $n \in \N$, there exist $P_n \in \PT$ and $Q_n \in \PTp$ satisfying $\|1-P_n g \|_{A^{p_n}} < \varepsilon_n$ and $\|g-zQ_n g \|_{A^{p_n}} < \varepsilon_n$. Since for all $p > 1$, there exists $N$ such that for all $n \geq N$ we have $p > p_n$,
$$\|1-P_n g \|_{A^p} \leq \|1-P_n g \|_{A^{p_n}} < \varepsilon_n \to  0,\qquad n\to \infty$$
and
$$\|g-zQ_n g \|_{A^p} \leq \|g-zQ_n g \|_{A^{p_n}} < \varepsilon_n \to  0,\qquad n\to \infty.$$
By Proposition \ref{base}, $g$ is cyclic in $A^p(\T)$ for all $p > 1$ and thus we have $g \in \mathcal{G}(K)$. The other inclusion is clear.

\bigskip \noindent
$(ii)$ follows from the fact that, for $f \in A(\T)$ and $g \in A^p(\T)$, $$\|fg\|_{A^p} \leq \|f\|_{A^1} \|g\|_{A^p}.$$

\bigskip \noindent
$(iii)$: Let $g_0 \in \mathcal{I}(K)$. We will show that for all $\eta > 0$, there exists $g \in G(\varepsilon,p)$ such that $\|g-g_0\|_{A^1} < \eta$. Without loss of generality, we suppose that $1 \in K$. So we have $g_0 \in \mathcal{I}(\{ 1 \})$. Since $K$ is a Helson set, we consider $\delta_2$ the constant satisfying the condition $(iii)$ of Definition \ref{Helson}. Since the set of cyclic vectors in $\mathcal{I}(\{ 1 \})$ is dense in $\mathcal{I}(\{ 1 \})$ (see \cite{ABA}), there exists a function $h$ cyclic in $\mathcal{I}(\{ 1 \})$ such that
$$\|h-g_0\|_{A^1} < \frac{\delta_2}{1+\delta_2} \eta.$$
Hence for all $t \in K$, 
$$|h(t)| = |h(t) - g_0(t)| \leq \|h-g_0\|_\infty \leq \|h-g_0\|_{A^1} < \frac{\delta_2}{1+\delta_2} \eta.$$
Since $h$ is cyclic in $\mathcal{I}(\{ 1 \})$, by Lemma \ref{cycliciteA0Ap}, $h$ is also cyclic in $\A{p}$. By Proposition \ref{base}, there exist $P \in \PT$ and $Q \in \PTp$ satisfying $\|1-Ph\|_{A^{p}} < \frac{\varepsilon}{2}$ and $\|h-zQh\|_{A^{p}} < \frac{\varepsilon}{2}$.
By Lemma \ref{lemme10}, there exists $f \in \A{}$ satisfying
$$\left\{ 
\begin{array}{rcl} f|_K &=&\displaystyle  h|_K ,\\ 
\|f\|_{A^1} &\leq& \displaystyle  \frac{ \|h|_K\|_\infty}{\delta_2} < \frac{\eta}{1+\delta_2} ,\\ 
\|f\|_{A^{p}} &<&\displaystyle  \min\Big( \frac{\varepsilon}{~~2\|P\|_{A^1}}, \frac{\varepsilon}{2\|1-zQ\|_{A^1}} \Big). \end{array}
 \right. $$
Let $g =h-f$. We have $g \in \mathcal{I}(K)$ since $f|_K = h|_K$. Moreover $g \in G(\varepsilon,p)$ since
$$\|1-Pg\|_{A^{p}} \leq \|1-Ph\|_{A^{p}} + \|P\|_{A^1} \|f\|_{A^{p}} < \varepsilon$$
and
$$\|g-zQg\|_{A^{p}} \leq  \|h-zQh\|_{A^{p}} + \|1-zQ\|_{A^1} \|f\|_{A^{p}} < \varepsilon.$$
Finally we have
$$\|g-g_0\|_{A^1} \leq  \|h - g_0\|_{A^1} + \|f\|_{A^1} < \frac{\delta_2}{1+\delta_2} \eta + \frac{\eta}{1+\delta_2} < \eta.$$
Hence $G(\varepsilon,p)$ is a dense subset of $\mathcal{I}(K)$.

\bigskip \noindent
By the Baire Category Theorem, $\underset{n=1}{\overset{\infty}{\bigcap}} G(\varepsilon_n,p_n)$ is a dense $G_\delta$ subset of $\mathcal{I}(K)$. In particular $\mathcal{G}(K)$ is not empty.
\end{proof}

\begin{proof} (of Theorem \ref{levmainresult})
We recall that $1<p<2$. Let $K$ be a Helson set given by Theorem \ref{LevHelson}. By Theorem \ref{Existence}, there exists $g \in A(\T)$ such that $g$ vanishes on $K$ and such that $g$ is cyclic in $A^p(\T)$. This implies in particular that $\log |g| \not \in L^1(\T)$. Moreover there exists $f \in C^1(\T)$ such that $\cZ(f)=\cZ(g)$ and $\log |f| \not \in L^1(\T)$ (take $f(\zeta)=e^{-1/d(\zeta,\cZ(g))}$ if $\zeta \not \in \cZ(g)$ and $0$ otherwise). By Proposition \ref{regCycl} and Theorem \ref{LevHelson}, $f$ is not bicyclic in $A^p(\T)$ so it is not cyclic in $A^p(\T)$.
\end{proof}

\begin{rem}
In the construction of $f$ and $g$ in Theorem \ref{levmainresult}, one is bicyclic and the other one is not.
A natural question arises: Is it possible to construct two bicyclic functions with the same properties as in Theorem \ref{levmainresult}?
\end{rem}

\section{Cyclicity in weighted $\ell^p$ spaces}

In the following, we denote $q=p/(p-1)$. In this section, we give both necessary conditions and sufficient conditions for $f$ to be cyclic in $A^p_\beta(\T)$, when $f$ is smooth, for $1<p<2$ and $0 \leq \beta q \leq 1$. When $\beta q > 1$, there is no cyclic vector in $A^p_\beta(\T)$. Indeed, for $1 \leq p < \infty$ and $\beta \geq 0$, $A^p_\beta(\T)$ is a Banach algebra of continuous functions if and only if $\beta q > 1$ (see \cite{ZAR}). So, if $f$ is cyclic in $A^p_\beta(\T)$, when $\beta q > 1$, then $\cZ(f)=\varnothing$ and $\log |f| \not \in L^1(\T)$ which is impossible since $f$ is continuous. In particular the conditions obtained in this section allow us to construct explicit examples of cyclic vectors. We also study the density of the set of cyclic vectors in weighted $\ell^p(\Z)$.

\subsection{Conditions for the cyclicity}

We denote by $H^\infty$ the space of functions $f \in L^\infty(\T)$ satisfying $\widehat{f}(n)=0$ for all $n < 0$. It is well known that $H^\infty$ can be identified with the space of bounded holomorphic functions on the disk $\mathbb{D} = \{ z \in \C,~ |z| < 1 \}$. Recall that outer functions are given by  
$$f(z)=\exp \int_\T\frac{\zeta+z}{\zeta-z}\log \varphi(\zeta) \frac{|d\zeta|}{2\pi}, \qquad z\in \mathbb{D},$$
where $\varphi$ is a positive function such that $\log \varphi\in L^1(\T)$. We denote by $f(\zeta)$ the radial limit of $f$ at $\zeta \in \T$ if it exists. We have $|f|=\varphi$  a.e. on $\T$. Moreover if $\varphi \in L^\infty(\T)$ then $f \in H^\infty$.

\bigskip

\begin{prop} \label{infimum}
Let $1<p<2$, $\beta \geq 0$ and $f \in A^2_\alpha(\T)$ with $\alpha > 1/p-1/2+\beta$. The following are equivalent.
\begin{enumerate}
\item $[f]_{\N_0}^{A^p_\beta(\T)}=[f]_\Z^{A^p_\beta(\T)}$,
\item $\inf \left\{ \|Pf\|_{A^p_\beta},~ P \in H^\infty,~ Pf \in A^2_\alpha(\T),~ P(0)=1 \right\}=0$.
\end{enumerate}
\end{prop}

Note that, by Lemma \ref{inclusionAp}, $\alpha > 1/p-1/2+\beta$ implies that $A^2_\alpha(\T) \subset A^p_\beta(\T)$.

\begin{proof}
We suppose first $[f]_{\N_0}^{A^p_\beta(\T)}=[f]_\Z^{A^p_\beta(\T)}$. By Proposition \ref{base}, there exists $P_n \in \PTp$ such that $\lim \|f-zP_nf\|_{A^p_\beta} = 0$ as $n \to \infty$. Since $1-zP_n \in H^\infty$, we have $$\inf \left\{ \|Pf\|_{A^p_\beta},~ P \in H^\infty,~ Pf \in A^2_\alpha(\T),~ P(0)=1 \right\}=0.$$
Now suppose that \textit{(2)} holds. Let $P_n \in H^\infty$, for $n \in \N$, such that $P_n(0)=1$, $P_nf \in A^2_\alpha(\T)$ and $\lim \|P_nf\|_{A^p_\beta} = 0$ as $n \to \infty$. We write $P_n=1-zQ_n$ with $Q_n \in H^\infty$. So we have, $\lim \|zQ_nf - f \|_{A^p_\beta} = 0$. Recall that $[f]_{\N_0}^{A^p_\beta(\T)}=[f]_\Z^{A^p_\beta(\T)}$ holds if and only if $f \in [f]_{\N}^{A^p_\beta(\T)}$. Thus it suffices to show that $zQ_nf \in [f]_{\N}^{A^p_\beta(\T)}$ since $[f]_{\N}^{A^p_\beta(\T)}$ is a closed subset of $A^p_\beta(\T)$. Note that $zQ_nf \in A^2_\alpha(\T)$ since $P_nf$ and $f$ are in $A^2_\alpha(\T)$. So we will show the following claim: if $Q \in H^\infty$ verifies $zQf \in A^2_\alpha(\T)$ then $zQf \in [f]_{\N}^{A^p_\beta(\T)}$.

We note, for $0<r<1$, $Q_r(z)=Q(rz)$. We have $zQ_rf \in [f]_{\N}^{A^p_\beta(\T)}$ since for all $\varepsilon > 0$, there exists $\tilde{Q_r} \in \PTp$ such that $\|Q_r - \tilde{Q_r}\|_{A^1_\beta}  < \varepsilon$ and $\| zQ_rf - z\tilde{Q_r}f \|_{A^p_\beta} \leq \|Q_r - \tilde{Q_r}\|_{A^1_\beta}\|zf\|_{A^p_\beta} < \varepsilon \|zf\|_{A^p_\beta} $. Moreover, by the Hahn-Banach theorem, $[f]_{\N}^{A^p_\beta(\T)}$ is weakly closed in $A^p_\beta(\T)$. So it suffices to show that $zQ_rf \rightarrow zQf$ weakly in $A^p_\beta(\T)$ as $r \to 1$. Let $g \in A^q_{-\beta}(\T)$ with $1/p+1/q=1$. There exists a sequence $(g_n)$ in $\PT$ such that $\lim \| g-g_n \|_{A^q_{-\beta}} = 0$, as $n \to \infty$. So we have
\begin{eqnarray*}
|\langle zQ_rf - zQf, g \rangle| &\leq& |\langle zQ_rf - zQf, g-g_n \rangle| + |\langle zQ_rf - zQf, g_n \rangle|\\
& \leq & \| zQ_rf - zQf \|_{A^p_\beta} \| g-g_n \|_{A^q_{-\beta}} +\\ & &~~~ \left| \int_\T (zQ_rf - zQf)(\zeta) g_n(\zeta) | d \zeta | \right|.
\end{eqnarray*}
Since $\| zQ_rf - zQf \|_{A^p_\beta} \leq \| zQ_rf - zQf \|_{A^2_\alpha}$, we can show, with the same calculus that in \cite[proposition 3.4]{RRS} and \cite{EKR}, that $\| zQ_rf - zQf \|_{A^p_\beta}$ remains bounded as $r \to 1$. Moreover $\| g-g_n \|_{A^q_{-\beta}} \rightarrow 0$ and, by the dominated convergence theorem, $\int_\T | (zQ_rf - zQf)(\zeta) g_n(\zeta) | | d \zeta | \rightarrow 0$ as $r \to 1$ since $f \in L^1(\T)$. So $zQ_rf \rightarrow zQf$ weakly in $A^p_\beta(\T)$ as $r \to 1$ and $zQf \in [f]_{\N}^{A^p_\beta(\T)}$. This completes the proof.
\end{proof}

We also need the following lemma, see the proof of Theorem 1.3 in \cite{KEL}.

\begin{lemme} \label{lemmeKEL}
Let $E$ be a closed subset of $\T$ with Lebesgue measure zero and such that $\log (d(\cdot,E)) \not \in L^1(\T)$.
Let $\delta' > 1/2$ and $\gamma >0$ such that $2\delta'-\gamma-1 \geq 0$. Let $F_\varepsilon$ be the outer function $F_\varepsilon$ satisfying
$$|F_\varepsilon(\zeta)| = \left( d(\zeta,E)^\gamma + \varepsilon \right)^{1/2} \qquad a.e. \text{ on } \T$$
and
$$M_\varepsilon = \frac{1}{2} \int_\T \log  \left( \frac{1}{d(\zeta,E)^\gamma+\varepsilon} \right) d\zeta.$$
For all $\varepsilon > 0$ we have
$$\iint_{\T^2}  d(\zeta',E)^{2(\delta'-\gamma)} \frac{ |F_\varepsilon(\zeta)-F_\varepsilon (\zeta')|^2}{|\zeta-\zeta'|^2} | d\zeta |  | d \zeta' | \lesssim M_\varepsilon.$$
\end{lemme}

The following result was obtained in \cite{KEL} for $p=2$ and $\beta=1/2$.

\begin{theo}
Let $1 < p < 2$ and $\beta > 0$ such that $\beta q < 1$. Let $f \in \text{Lip}_\delta(\T)$, where $\delta > \beta + 1/p - 1/2$. We suppose that $\cZ(f)$ has Lebesgue measure zero.  If 
$$\int_\T \log d(\zeta,\cZ(f)) |d\zeta|=-\infty$$
 then 
 $$[f]_\Z^{A^p_\beta(\T)}=[f]_{\N_0}^{A^p_\beta(\T)}.$$
  Furthermore,  if $f \in A^1_\beta(\T)$ and $\dim(\cZ(f)) < \frac{2}{q}(1 - \beta q)$ then $f$ is cyclic in $A^p_\beta(\T)$.
\end{theo}

\begin{proof}
Let $\alpha > 0$ such that $\beta + 1/p - 1/2 < \alpha < \min(\delta,1/2)$. First we have $f \in \text{Lip}_\delta(\T) \subset A^2_\alpha(\T) \subset  A^p_\beta(\T)$.
Indeed by Douglas' formula (see \cite[section 5]{RRS}), $f \in A^2_\alpha(\T)$ if and only if $f \in L^2(\T)$ and 
$$\D_\alpha(f) = \iint_{\T^2} \frac{|f(\zeta)-f(\zeta')|^2}{|\zeta-\zeta'|^{1+2\alpha}} |d\zeta||d \zeta'| < \infty.$$
Moreover we have
\begin{equation} \label{Douglasform}
\|f\|_{A^2_\alpha}^2 \asymp \|f\|_{L^2(\T)}^2 + \D_\alpha(f).
\end{equation}
Thus if $f \in \text{Lip}_\delta(\T)$ then 
$$\D_\alpha(f) \lesssim  \iint_{\T^2} \frac{1}{|\zeta-\zeta'|^{1+2\alpha-2\delta}} |d\zeta||d \zeta'| < \infty$$
and we get $f \in A^2_\alpha(\T)$. On the other hand, by Proposition \ref{inclusionAp}, we have $A^2_\alpha(\T) \subset  A^p_\beta(\T)$.

By Proposition \ref{infimum}, it suffices to show that 
$$\inf \left\{ \|Pf\|_{A^p_\beta},~ P \in H^\infty,~ Pf \in A^2_\alpha(\T),~ P(0)=1 \right\}=0$$
to complete the proof.
 Let $\varepsilon > 0$ and $\gamma>0$ such that $\gamma < \delta - \alpha$. We define $p_\varepsilon$ to be the outer function satisfying
$$|p_\varepsilon(\zeta)| = \frac{e^{-M_\varepsilon}}{(d(\zeta,E)^\gamma+\varepsilon)^{1/2}} \qquad a.e.$$
where $E=\cZ(f)$ and 
$$M_\varepsilon = \frac{1}{2} \int_\T \log  \left( \frac{1}{d(\zeta,E)^\gamma+\varepsilon} \right) |d\zeta|$$
so that $p_\varepsilon(0)=1$.  Since $\log d(\cdot,E)\notin L^1(\T)$,  $M_\varepsilon \to \infty$ as $\varepsilon \to 0$.

\bigskip

We will show that
$$\|p_\varepsilon f \|_{A^p_\beta} \leq \|p_\varepsilon f \|_{A^2_\alpha} \longrightarrow 0 ~~~~ \text{ as } \varepsilon \to 0.$$
By \eqref{Douglasform}, we have 
$$\|p_\varepsilon f \|_{A^2_\alpha}^2 \asymp \|p_\varepsilon f\|_{L^2(\T)}^2 + \D_\alpha(p_\varepsilon f).$$
First, since $\gamma < 2\delta$,
$$\|p_\varepsilon f \|_{L^2(\T)}^2 \lesssim e^{-2M_\varepsilon} \int_\T \frac{d(\zeta,E)^{2\delta}}{d(\zeta,E)^{\gamma}} |d \zeta| \lesssim e^{-2M_\varepsilon} \to 0.$$
Then we have
$$\D_\alpha(p_\varepsilon f) = \iint_{\T^2} \frac{|p_\varepsilon f(\zeta)-p_\varepsilon f(\zeta')|^2}{|\zeta-\zeta'|^{1+2\alpha}} |d\zeta| |d \zeta'|.$$
Note that for all $(\zeta,\zeta') \in \T^2$ we have
$$|p_\varepsilon f(\zeta)-p_\varepsilon f(\zeta')|^2 \leq 2 |p_\varepsilon (\zeta)|^2 |f(\zeta)- f(\zeta')|^2 + 2 |f(\zeta')|^2|p_\varepsilon(\zeta)-p_\varepsilon (\zeta')|^2.$$
So if we denote $\Gamma = \{ (\zeta,\zeta') \in \T^2,~ |d(\zeta',E)| \leq |d(\zeta,E)| \}$, we obtain
\begin{eqnarray*}
\D_\alpha(p_\varepsilon f) & \lesssim & \iint_\Gamma |p_\varepsilon (\zeta)|^2 \frac{ |f(\zeta)- f(\zeta')|^2}{|\zeta-\zeta'|^{1+2\alpha}} |d\zeta| |d \zeta'| + \iint_\Gamma |f(\zeta')|^2 \frac{ |p_\varepsilon(\zeta)-p_\varepsilon (\zeta')|^2}{|\zeta-\zeta'|^{1+2\alpha}} |d\zeta| |d \zeta'|\\
& = & A_\varepsilon + B_\varepsilon.
\end{eqnarray*}
Let $\eta$ such that $\frac{\gamma}{2\delta} \leq \eta < \frac{2\delta-2\alpha}{2\delta}$. Then, using the fact that for $(\zeta,\zeta') \in \Gamma$, $|f(\zeta')| \leq d(\zeta,E)^\delta$, we get
\begin{eqnarray*}
A_\varepsilon & = & e^{-2M_\varepsilon} \iint_\Gamma \left( \frac{ |f(\zeta)- f(\zeta')|^{2\eta}}{d(\zeta,E)^\gamma+\varepsilon} \right) \left( \frac{ |f(\zeta)- f(\zeta')|^{2-2\eta}}{|\zeta-\zeta'|^{1+2\alpha}} \right) |d \zeta| |d \zeta'|\\
& \lesssim & e^{-2M_\varepsilon} \iint_{\T^2} \frac{ d(\zeta,E)^{2\eta \delta}}{d(\zeta,E)^\gamma} \left( \frac{ |f(\zeta)- f(\zeta')|^{2-2\eta}}{|\zeta-\zeta'|^{1+2\alpha}} \right) |d \zeta| |d \zeta'|\\
& \lesssim & e^{-2M_\varepsilon} \iint_{\T^2} \frac{ |\zeta- \zeta'|^{(2-2\eta)\delta}}{|\zeta-\zeta'|^{1+2\alpha}} |d \zeta| |d \zeta'| \\
&\lesssim & e^{-2M_\varepsilon} \to 0.
\end{eqnarray*}
To estimate $B_\varepsilon$, we consider the outer function $F_\varepsilon$ given in Lemma \ref{lemmeKEL}. Let
$$\Gamma_1 = \{ (\zeta,\zeta') \in \Gamma,~ d(\zeta',E) \leq |\zeta-\zeta'|\},$$
$$\Gamma_2 = \{ (\zeta,\zeta') \in \Gamma,~ d(\zeta',E) > |\zeta-\zeta'|\}.$$
We have
\begin{eqnarray*}
B_\varepsilon &=&   \iint_\Gamma 2 |f(\zeta')|^2 |p_\varepsilon(\zeta)p_\varepsilon (\zeta')|^2 \frac{ |1/p_\varepsilon(\zeta')-1/p_\varepsilon(\zeta)|^2}{|\zeta-\zeta'|^{1+2\alpha}} |d\zeta| |d \zeta'| \\
& \lesssim & e^{-2M_\varepsilon} \iint_\Gamma \frac{d(\zeta',E)^{2\delta}}{(d(\zeta,E)^\gamma+\varepsilon)(d(\zeta',E)^\gamma+\varepsilon)} \frac{ |F_\varepsilon(\zeta)-F_\varepsilon (\zeta')|^2}{|\zeta-\zeta'|^{1+2\alpha}} |d\zeta| |d \zeta'|\\
& \lesssim & e^{-2M_\varepsilon} \iint_{\Gamma_1} |\zeta - \zeta'|^{2\delta-2\gamma-2\alpha-1} |F_\varepsilon(\zeta)-F_\varepsilon (\zeta')|^2 |d\zeta| |d \zeta'|\\
& & ~~ + e^{-2M_\varepsilon} \iint_{\Gamma_2} d(\zeta',E)^{2\delta-2\gamma-2\alpha+1} \frac{ |F_\varepsilon(\zeta)-F_\varepsilon (\zeta')|^2}{|\zeta-\zeta'|^2} |d\zeta| |d \zeta'|\\
&=& B_\varepsilon^1 + B_\varepsilon^2.
\end{eqnarray*}
Since $\gamma < \delta - \alpha$ and $|F_\varepsilon|$ is bounded on $\T$, we have $B_\varepsilon^1 \lesssim e^{-2M_\varepsilon} \to 0$ as $\varepsilon \to 0$. On the other hand, applying Lemma \ref{lemmeKEL} with $\delta'=\delta-\alpha+1/2$, we obtain $B_\varepsilon^2 \lesssim M_\varepsilon e^{-2M_\varepsilon}$. Thus $B_\varepsilon^2 \to 0$ as $\varepsilon \to 0$.
Therefore we have $\lim \|p_\varepsilon f \|_{A^p_\beta} =0$ and, by Proposition \ref{infimum}, $[f]_\Z^{A^p_\beta(\T)}=[f]_{\N_0}^{A^p_\beta(\T)}$. Furthermore, by \cite[Theorem A]{LMF}, if $\dim(\cZ(f)) < \frac{2}{q}(1 - \beta q)$ then $f$ is bicyclic in $A^p_\beta(\T)$, so $f$ is cyclic in $A^p_\beta(\T)$.
\end{proof}

\subsection*{Remark} Let $1 < p < 2$ and $\beta > 0$ such that $\beta q < 1$. Using the fact that $A^2_{1/2}(\T) \subset A^p_\beta(\T)$ and \cite[Theorem 1.2]{KEL} we obtain, for $f \in A^p_\beta(\T)$ such that $|f| \in C^1(\T)$ and $|f|' \in \Lip_\delta(\T)$ for $\delta > 0$, the following result: if $\dim(\cZ(f)) < \frac{2}{q}(1 - \beta q)$ and $\log |f| \not \in L^1(\T)$ then $f^2$ is cyclic in $A^p_\beta(\T)$.

\subsection{Cyclicity for smooth functions}\label{sub42}

When $f \in C^\infty(\T)$, Makarov proves in \cite{MAK}, that $[f]_{\N_0}^{C^\infty(\T)}=[f]_\Z^{C^\infty(\T)}$ if and only if $\log |f| \not \in L^1(\T)$. As a consequence of Makarov's Theorem and some results obtained in \cite{LMF}, we get the following result:

\begin{theo} \label{Cinfty}
Let $1<p<2$, $\beta \geq 0$ such that $\beta q \leq 1$.
We have the following assertions:
\begin{enumerate}
\item If $f \in C^\infty(\T)$, $\dim(\cZ(f)) < \frac{2}{q}(1-\beta q)$ and $\log |f| \not \in L^1(\T)$ then $f$ is cyclic in $A^p_\beta(\T)$.
\item If $f \in A^1_\beta(\T)$, $\dim(\cZ(f)) > 1-\beta q$ then $f$ is not cyclic in $A^p_\beta(\T)$.
\item For $\frac{2}{q}(1-\beta q) \leq \alpha \leq 1$, there exists a closed subset $E \subset \T$ such that $\dim(E) = \alpha$ and every $f \in A^1_\beta(\T)$ satisfying $\cZ(f) \subset E$ is not cyclic in $A^p_\beta(\T)$.
\item If $p=\frac{2k}{2k-1}$ for some $k \in \N$ then there exists a closed set $E \subset \T$ such that $\dim(E)=1-\beta q$ and every $f \in C^\infty(\T)$ satisfying $\cZ(f) \subset E$ and $\log |f| \not \in L^1(\T)$ is cyclic in $A^p_\beta(\T)$.
\end{enumerate}
\end{theo}

Note that, for $1<p<2$ and $\beta \geq 0$ such that $\beta q \leq 1$, if $f$ is cyclic in $A^p_\beta(\T)$ then $\log |f| \not \in L^1(\T)$.

\begin{proof}
For $(2)$ and $(3)$, we use \cite[Theorem 3.4 and Theorem 4.3]{LMF} since if $f$ is not bicyclic then $f$ is not cyclic in $A^p_\beta(\T)$.

For $(1)$ and $(4)$, suppose $\log |f| \not \in L^1(\T)$. By \cite{MAK}, we have $[f]_{\N_0}^{C^\infty(\T)} = [f]_{\Z}^{C^\infty(\T)}$. Since the embedding $C^\infty(\T) \subset A^p_\beta(\T)$ is continuous, we obtain $[f]_{\N_0}^{A^p_\beta(\T)} = [f]_{\Z}^{A^p_\beta(\T)}$. Moreover, by \cite[Theorem 3.4]{LMF}, $f$ is bicyclic in $A^p_\beta(\T)$. Therefore $f$ is cyclic in $A^p_\beta(\T)$.
\end{proof}

In the following, we denote, for $E \subset \T$, $|E|$ the Lebesgue measure of $E$.

\subsection*{Examples} Let $E$ be a closed subset of $\T$ and $\Lambda$ be a continuous, positive, decreasing function defined on $(0,\infty)$ such that $\lim \Lambda(t) = \infty$ and $\lim t \Lambda(t) = 0$ as $t \to 0$. We consider the function $f$ defined on $\T$ by
\begin{equation}\label{fonctionf}
f(\zeta) = \exp \left( - \Lambda( d(\zeta,E) ) \right)
\end{equation}

if $\zeta \not \in E$ and $f(\zeta)=0$ otherwise. So, by \cite[Proposition A.1]{KEL2}, we have
$$\int_\T \log |f(\zeta)| |d \zeta| = - \int_\T \Lambda( d(\zeta,E) ) |d \zeta| =  - \infty \Leftrightarrow \int_\T |E_t| d\Lambda(t) = -\infty,$$
where $E_t=\{\zeta \in \T,~ d(\zeta,E) \leq t \}$. Let $N_E(t)$ be the smallest number of closed arcs of length $2t$ that cover $E$. We have, by \cite[Lemma A.3]{KEL2}, that $tN_E(t) \leq |E_t| \leq 4tN_E(t)$. Thus, by \cite[Theorem 2, p. 30]{CAR}, we have for $0 < \alpha < 1$,
$$\int_\T \frac{1}{t^{\alpha+1}N_E(t)} dt = \infty \Leftrightarrow \int_\T \frac{1}{t^{\alpha}|E_t|} dt = \infty \Rightarrow \dim(E) < \alpha.$$
Now we suppose that $\Lambda(t)=1/t^\gamma$ for $\gamma >  0$. The function $f$ defined by \eqref{fonctionf} is infinitely differentiable on $\T$. Let $1<p<2$, $\beta \geq 0$ such that $\beta q \leq 1$.
We can construct a Cantor set $E$ such that $$\int_\T \frac{|E_t|}{t^{\gamma+1}} dt = \infty ~\text{ and }~ \int_\T \frac{1}{t^{\alpha}|E_t|} dt = \infty$$ for $\alpha < \frac{2}{q}(1-\beta q)$. Then, by Theorem \ref{Cinfty}, the function $f$ is cyclic in $A^p_\beta(\T)$.
Moreover if $\gamma \geq 1$ we have $\log |f| \not \in L^1(\T)$. So, by Theorem \ref{Cinfty},

\begin{enumerate}
\item if $\dim(E) < \frac{2}{q}(1-\beta q)$ then $f$ is cyclic in $A^p_\beta(\T)$;
\item if $\dim(E) > 1-\beta q$ then $f$ is not cyclic in $A^p_\beta(\T)$;
\item for $\frac{2}{q}(1-\beta q) \leq \alpha \leq 1$, there exists a closed subset $E \subset \T$ such that $\dim(E) = \alpha$ and $f$ is not cyclic in $A^p_\beta(\T)$;
\item if $p=\frac{2k}{2k-1}$ for some $k \in \N$ then there exists a closed set $E \subset \T$ such that $\dim(E)=1-\beta q$ and $f$ is cyclic in $A^p_\beta(\T)$.
\end{enumerate}

\subsection{Density of cyclic vectors set}

When $\beta q \leq 1$, we have seen, in subsection \ref{sub42}, that there exist cyclic vectors in $A^p_\beta(\T)$. Here we show that the set of cyclic vectors in $A^p_\beta(\T)$ is a dense $G_\delta$ subset of $A^p_\beta(\T)$.

\begin{prop} \label{zerofini} Let $f \in A(\T)$.
If $\cZ(f)$ is finite then $f$ is bicyclic in $A^p_\beta(\T)$ for $p > 1$ and $\beta \geq 0$ such that $\beta q \leq 1$.
\end{prop}

\begin{proof}
Let $S \in A^q_{-\beta}(\T)$ such that $\langle S,z^nf \rangle=0$ for all $n \in \Z$. It suffices to prove that $S=0$.
First we have $\supp(S) \subset \cZ(f)$ (see \cite[Lemma 2.4]{LMF}) and since $\cZ(f)$ is finite, we obtain that
$$S = \sum_{k=1}^N \sum_{n=0}^{M} \lambda_{k,n}~ \delta_{z_k}^{(n)}$$
where $\cZ(f) = \{ z_k,~ k\in \llbracket 1, N\rrbracket \} \subset \T$, $\lambda_{k,n} \in \C$ and where $\delta_{z_k}^{(n)}$ is the $n$th derivative of the Dirac delta measure concentrated at the point $z_k$.

Suppose that $S \neq 0$. There exist $k_0$ and $n_0$ such that $\lambda_{k_0,n_0} \neq 0$. We define
$$T = \prod_{\substack{k=1\\k \neq k_0}}^N (z - z_k)^{M+1} ~ S.$$
So $T \in A^q_{-\beta}(\T)$ and $T \neq 0$.
Hence there exists $(c_n) \in \C^{M+1}$ such that for every test function $\varphi \in \mathcal{D}(\T)$,
$$\langle T,\varphi \rangle = \langle S, \prod_{\substack{k=1\\k \neq k_0}}^N (z - z_k)^{M+1} \varphi \rangle = \sum_{n = 0}^M \lambda_{k_0,n} ~ c_n ~ \varphi^{(n)}(z_{k_0}).$$
If we denote $n_1 = \max\{ n,~ c_n \lambda_{k_0,n} \neq 0 \}$, then we have for all $l \in \Z$
$$|\langle T,z^l \rangle |  = \left|  \sum_{n = 0}^N \lambda_{k_0,n} ~ c_n ~ (il)^n z_{k_0}^l \right| ~ \underset{l \to \infty}{\sim} ~ |\lambda_{k_0,n_1}| ~|c_{n_1}| |l|^{n_1}$$
(here $z^l=e^{il\theta}$).
This is not possible since  $T \in A^q_{-\beta}(\T)$ and $\beta q \leq 1$.
\end{proof}

\begin{lemme} \label{lemmedense}
Let $1 < p \leq 2$ and $\beta \geq 0$ such that $\beta q \leq 1$ where $1/p+1/q=1$. If $f \in A^p_\beta(\T)$ is cyclic in $A^p_\beta(\T)$ and if $P \in \PT$ is not zero then $Pf$ is cyclic in $A^p_\beta(\T)$.
\end{lemme}

\begin{proof}
Let $\varepsilon > 0$. By Lemma \ref{zerofini}, $P$ is bicyclic in $A^p_\beta(\T)$ since $\cZ(P)$ is finite. So there exists $P_1 \in \PT$ such that $\| 1-P_1P\|_{A^p_\beta} < \varepsilon$. Moreover since $f$ is bicyclic, there exists $P_2 \in \PT$ such that
$$\| 1-P_2f \|_{A^p_\beta} < \frac{\varepsilon}{\| P_1 P \|_{A^1_\beta}}.$$
Therefore we have 
$$\| 1-P_1P_2 Pf \|_{A^p_\beta} \leq \| 1-P_1P\|_{A^p_\beta} + \| P_1 P \|_{A^1_\beta} \| 1-P_2f \|_{A^p_\beta} < 2 \varepsilon.$$
So $Pf$ is bicyclic in $A^p_\beta(\T)$.

Moreover, since $f$ is cyclic in $A^p_\beta(\T)$, there exists $Q \in \PTp$ such that $\| f - zQf\|_{A^p_\beta} < \frac{\varepsilon}{ \| P \|_{A^1_\beta}}$. So we have
$$\| Pf - zQPf\|_{A^p_\beta} \leq \| P \|_{A^1_\beta} \| f - zQf \|_{A^p_\beta} < \varepsilon.$$
Hence, by Proposition \ref{base}, $Pf$ is cyclic in $A^p_\beta(\T)$.
\end{proof}

\begin{theo}
Let $1 < p \leq 2$ and $\beta \geq 0$ such that $\beta q \leq 1$ where $1/p+1/q=1$.
The set of cyclic vectors in $A^p_\beta(\T)$ is a dense $G_\delta$ subset of $A^p_\beta(\T)$.
\end{theo}

\begin{proof}
Let $\varepsilon > 0$. We consider the set
$$
G_\varepsilon = \{ g \in A^p_\beta(\T),~ \exists P \in \PT,~ \exists Q \in \PTp,~ \|1-Pg\|_{A^p_\beta} < \varepsilon ~\text{ and }~ \|g-zQg\|_{A^p_\beta} < \varepsilon \}.
$$
To use a Baire Category argument, we will show that
$$\begin{array}{ll} (i) & \underset{n=1}{\overset{\infty}{\bigcap}} G_{1/n} = \{ f \in A^p_\beta(\T) \text{ cyclic in } A^p_\beta(\T) \},  \\ (ii) & \text{for $\varepsilon > 0$},~~ G_\varepsilon \text{ is an open subset of } A^p_\beta(\T), \\ (iii) & \text{for $\varepsilon > 0$},~~ G_\varepsilon \text{ is a dense subset of } A^p_\beta(\T). \end{array}$$

\bigskip \noindent
$(i)$: Let $g \in A^p_\beta(\T)$ such that for all $n \in \N$, there exist $P_n \in \PT$ and $Q_n \in \PTp$ satisfying $\|1-P_n g \|_{A^p_\beta} < 1/n$ and $\|g-zQ_n g \|_{A^p_\beta} < 1/n$. By Proposition \ref{base}, $g$ is cyclic in $A^p_\beta(\T)$. The other inclusion is clear.

\bigskip \noindent
$(ii)$ follows from the fact that, for $f \in A^1_\beta(\T)$ and $g \in A^p_\beta(\T)$, $$\|fg\|_{A^p_\beta} \leq \|f\|_{A^1_\beta} \|g\|_{A^p_\beta}.$$

\bigskip \noindent
$(iii)$: Since for $\varepsilon > 0$, $\{ f \in A^p_\beta(\T) \text{ cyclic in } A^p_\beta(\T) \} \subset G_\varepsilon$, we will show that $\{ f \in A^p_\beta(\T) \text{ cyclic in } A^p_\beta(\T) \}$ is dense in $A^p_\beta(\T)$. Let $h \in A^p_\beta(\T)$ and $\eta > 0$ such that $\eta < \|h\|_{A^p_\beta}$. Consider $f \in A^p_\beta(\T)$ which is cyclic in $A^p_\beta(\T)$. So there exists a non-zero polynomial $P \in \PTp$ such that $\|h-Pf\|_{A^p_\beta} < \eta$ and by Lemma \ref{lemmedense}, $Pf$ is cyclic in $A^p_\beta(\T)$. Therefore this shows that $G_\varepsilon$ is dense in $A^p_\beta(\T)$.
\end{proof}

\section*{Acknowledgements}

I would like to acknowledge my doctoral advisors, K. Kellay and M. Zarrabi. I am grateful to the referee for his remarks and suggestions.

\end{document}